\documentclass[fleqn,preprint,3p]{elsarticle}

\usepackage{amsmath}
\usepackage{amssymb}
\usepackage{amsthm}

\usepackage{color}

\usepackage{graphicx}
\usepackage{subfigure}

\newtheorem{theorem}{Theorem}[section]
\newtheorem{lemma}[theorem]{Lemma}

\newdefinition{definition}[theorem]{Definition}
\newdefinition{remark}[theorem]{Remark}
\newdefinition{example}[theorem]{Example}

\begin{document}

\title{Some solutions of  functional equation $I(I(y,x), I(x,y))=I(x,y)$ involving fuzzy implications }
\author{Nageswara Rao Vemuri}
\ead{nrvemuriz@uohyd.ac.in}

\address{School of Mathematics and Statistics\\ University of Hyderabad \\
Hyderabad - 500 046, INDIA}

\begin{frontmatter}
	
\begin{abstract}
In this article, a functional equation(IE) involving fuzzy implications has been considered. Two different perspectives of this equation have been provided to realize its significance. As it is very difficult to find the solutions of (IE) in general, the investigation of solutions of (IE) is restricted to main families of fuzzy implications.

\end{abstract}

\end{frontmatter}
%%%%%%%%%%%%%%%%%%%%%%%%%%%%%%%%%%%%%%%%%%%%%%%%%%%%%%%%%%%%%%%%%%%%%%%%%%%%%%%%%%%%%%%%%%%%%%%%%%%%%%%%%%%%%%%%
%%%%%%%%%%%%%%%%%%%%%%%%%%%%%%%%%%%%%%%%%%%%%%%%%%%%%%%%%%%%%%%%%%%%%%%%%%%%%%%%%%%%%%%%%%%%%%%%%%%%%%%%%%%%%%%%
\section{Introduction}
Fuzzy implications are a generalization of classical implication $\rightarrow$ from $\{0,1\}$ setting to the multivalued(MV) setting $[0,1]$.
They are defined as follows.
\begin{definition}[\cite{Baczynski_Jayaram_2008}, Definition 1.1.1 \& \cite{kitainik_1993, Fodor_Roubens_1994}]
\label{def_FI}
A function $I\colon [0,1]^2\to [0,1]$ is called a \emph{fuzzy implication} if it satisfies, for all $x,x_1,x_2,y,y_1,y_2\in[0,1]$, the following conditions:
\begin{gather}
    \label{eqn:I1}
    \text{if } x_1 \leq x_2, \text{ then } I(x_1,y) \geq I(x_2,y) \; , \text{ i.e., } I(\:\cdot\:,y) \text{ is decreasing\;,}
    \tag{I1}
\\
    \label{eqn:I2}
     \text{if } y_1 \leq y_2, \text{ then } I(x,y_1) \leq I(x,y_2) \; , \text{ i.e., } I(x,\:\cdot\:) \text{ is increasing\;,}
     \tag{I2}
\\
    \label{eqn:I3}
    I(0,0) =1 \; ,I(1,1)=1 \ , I(1,0)=0 \ .
    \tag{I3}
\end{gather}
The set of fuzzy implications will be denoted by $\mathbb{I}$.
\end{definition}
 Fuzzy implications play a significant role in the areas like, artificial intelligence, machine learning, decison making etc. The suitability of a fuzzy implication for a particular application/ problem will be decided by the properties or functional equations that the particular fuzzy implication satisfies.

Recall that the properties or functional equations of fuzzy implications are a mere generalization of classical tautologies to MV-setting and they involve some other fuzzy logic connectives also. However, functional equations of fuzzy implications are not only a generalizations of classical tautologies but may have also some mathematical perspectives. Thus the investigations of solutions of functional equations of fuzzy implications not only give fuzzy implications suitable for a specific application but also lead to glean mathematical aspects of the set of fuzzy implications.

Unfortunately, only few works have done on the investigations of solutions of functional equations of fuzzy implications so far, please see, \cite{Shi_Ruan_Kerre_2007}, \cite{Vemuri_Jayaram_2014a}, \cite{Vemuri_2015}.

In this article, we consider the following functional equation, for a fuzzy implication $I$,
\begin{equation}
\label{eqn:IE}
\tag{IE}
I(I(y,x),I(x,y))= I(x,y)
\end{equation}

\subsection{IE: Interpretation}
The functional equation \eqref{eqn:IE} has the following perspectives.
\begin{description}
\item[(A). Logical Perspective:]In the classical logic, the formula 
\begin{equation*}
((q \rightarrow p) \rightarrow (p \rightarrow q))\equiv (p\rightarrow q)
\end{equation*} where $\rightarrow$ is a classical implication, is always a tautology. This tautology, when generalized to fuzzy logic,  it can be represented by
\begin{equation}
\label{eqn:FE}
\tag{GFE}
I_{1}(I_{2}(y,x),I_{3}(x,y))= I_{4}(x,y)
\end{equation}
where $I_{1}, I_{2}, I_{3}$ and $I_{4}$ are fuzzy implications defined on $[0,1]$ and need not be the same always.
Note that it is never possible to find fuzzy implications $I_{1}, I_{2}, I_{3}$ and $I_{4}$ satisfying  \eqref{eqn:IE} simultaneously. Due to this reason, in this article, we restrict ourselves to the case of $I_{1}=I_{2}=I_{3}=I_{4}$, which gives us \eqref{eqn:FE}.
\item [(B). Algebraic Perspective:] In \cite{Vemuri_Jayaram_2012}, Vemuri and Jayaram proposed a generating method $\triangledown$ of fuzzy implications as follows.
\begin{definition} 
	Let $I,J \in \mathbb{I}$. Define the function $(I \triangledown J)$ as
\begin{equation*}
	(I \triangledown J)(x,y)=I(J(y,x),J(x,y)), \qquad x,y \in [0,1] \; . 
\end{equation*}
\end{definition}
Further, they have shown that $(I \triangledown J)$ is always a fuzzy implication and also, the pair $(\mathbb{I}, \triangledown)$ is a semigroup. The set of idempotent elements of the semigroup $(\mathbb{I}, \triangledown)$ is exactly the collection of fuzzy implications satisfying the functional equation \eqref{eqn:IE}. It is already known that the set of idempotent elements of a semigroup will glean more properties of that semigroup. 
\end{description}
Due to the perspectives of \eqref{eqn:IE}, it is always useful to investigate the solutions of \eqref{eqn:IE}.
%%%%%%%%%%%%%%%%%%%%%%%%%%%%%%%%%%%%%%%%%%%%%%%%%%%%%%%%%%%%%%%%%%%%%%%%%%%%%%%%%%%%%%%%%%%%%%%%%%%%%%%%%%%%%%%%
\subsection{Objectives of this work:}
%%%%%%%%%%%%%%%%%%%%%%%%%%%%%%%%%%%%%%%%%%%%%%%%%%%%%%%%%%%%%%%%%%%%%%%%%%%%%%%%%%%%%%%%%%%%%%%%%%%%%%%%%%%%%%%%
In this article, our objectives are as follows:
\begin{enumerate}
	\item To investigate the relation between the properties, functional equations and \eqref{eqn:IE}.
	\item To investigate the fuzzy implications satisfying \eqref{eqn:IE}.
\end{enumerate}
Here, in the following, we recall the basic properties of fuzzy implications which will be useful in the sequel.
\begin{definition}[cf.~\cite{Baczynski_Jayaram_2008}, Definition 1.3.1]
\label{def_basic_prop_imp}
\begin{itemize}
  \item A fuzzy implication $I$ is said to satisfy 
 \begin{enumerate}
 \renewcommand{\labelenumi}{(\roman{enumi})}
\item 
	the left neutrality property \eqref{eqn:NP} if 
\begin{align}
\tag{NP}	I(1,y)=y,  \qquad y \in [0,1] \; . \label{eqn:NP}
\end{align}
\item the ordering property \eqref{eqn:OP}, if 
\begin{align}
\tag{OP}  x \leq y \Longleftrightarrow I(x,y)=1 \; \ \qquad \ x,y \in [0,1] \ . \label{eqn:OP}
\end{align}
\item the identity principle \eqref{eqn:IP}, if 
\begin{align}
\tag{IP} I(x,x)=1,  \qquad  x \in [0,1] \; . \label{eqn:IP}
\end{align}
\item the exchange principle \eqref{eqn:EP}, if 
\begin{align}
\tag{EP} I(x,I(y,z))=I(y,I(x,z)),  \qquad x,y,z \in [0,1] \; . \label{eqn:EP}
\end{align}
\end{enumerate}
	\item A fuzzy implication $I$ is said to be continuous if it is continuous in both the variables. 
\end{itemize}
\end{definition}
%%%%%%%%%%%%%%%%%%%%%%%%%%%%%%%%%%%%%%%%%%%%%%%%%%%%%%%%%%%%%%%%%%%%%%%%%%%%%%%%%%%%%%%%%%%%%%%%%%%%%%%%%%%%%%%%
%%%%%%%%%%%%%%%%%%%%%%%%%%%%%%%%%%%%%%%%%%%%%%%%%%%%%%%%%%%%%%%%%%%%%%%%%%%%%%%%%%%%%%%%%%%%%%%%%%%%%%%%%%%%%%%%
\section{Some solutions of \eqref{eqn:IE}}
%%%%%%%%%%%%%%%%%%%%%%%%%%%%%%%%%%%%%%%%%%%%%%%%%%%%%%%%%%%%%%%%%%%%%%%%%%%%%%%%%%%%%%%%%%%%%%%%%%%%%%%%%%%%%%%%
Due to the variety of fuzzy implications, it is very difficult to find the fuzzy implications that satisfy \eqref{eqn:IE}. Hence, in this article, we restrict the solutions of \eqref{eqn:IE} to major families of fuzzy implications only. The more details about the t-norms, t-conorms and fuzzy negations can be found in \cite{Klement_Mesiar_Pap_2000} and \cite{Baczynski_Jayaram_2008}.
\begin{theorem}
\label{thm:OP_FE}
Let $I$ has (OP). Then  the following statements are equivalent:

\begin{enumerate}
 \renewcommand{\labelenumi}{(\roman{enumi})}
\item $I$ satisfies the equation (\ref{eqn:IE}).
\item $I$ has (NP) on its range.
\end{enumerate}
\end{theorem}

\begin{proof}
\begin{enumerate}
 \renewcommand{\labelenumi}{(\roman{enumi})}
\item $\Rightarrow (ii)$ Let $I$ has (OP) and satisfies the equation (\ref{eqn:IE}). Let $\alpha \in range \ of \ I.$ Then there exist $x_{0},y_{0} \in [0,1]$ such that $I(x_{0},y_{0})=\alpha$.
If $\alpha=1, $  then it is clear. Let $\alpha < 1.$ Since $ I$ has (OP) and $\alpha <1,$ we have $x_{0}> y_{0}$  and $I(y_{0},x_{0})=1$.
Now $I(1, \alpha)=I(I(y_{0},x_{0}),I(x_{0},y_{0}))=I(x_{0},y_{0})= \alpha.$ Thus $I$ has (NP) on the range.

 \item  $\Rightarrow (i)$ Let $I$ has (OP) and has (NP) on the range. If $ x \leq y,$ then $I(x,y)=1$ and $I(I(y,x),I(x,y))=I(I(y,x),1)=1.$ On the other hand let $x  >y$. Then 
$$I(I(y,x),I(x,y))=I(1,I(x,y))=I(x,y)$$ by our assumption. Thus $I$ satisfies the equation (\ref{eqn:IE}).
\end{enumerate}
\end{proof}

%%%%%%%%%%%%%%%%%%%%%%%%%%%%%%%%%%%%%%%%%%%%%%%%%%%%%%%%%%%%%%%%%%%%%%%%%%%%%%%%%%%%%%%%%%%%%%%%%%%%%%%%%%%%%%%%
%%%%%%%%%%%%%%%%%%%%%%%%%%%%%%%%%%%%%%%%%%%%%%%%%%%%%%%%%%%%%%%%%%%%%%%%%%%%%%%%%%%%%%%%%%%%%%%%%%%%%%%%%%%%%%%%
\subsection{R-implications and the equation (\ref{eqn:IE}).}
%%%%%%%%%%%%%%%%%%%%%%%%%%%%%%%%%%%%%%%%%%%%%%%%%%%%%%%%%%%%%%%%%%%%%%%%%%%%%%%%%%%%%%%%%%%%%%%%%%%%%%%%%%%%%%%%
%%%%%%%%%%%%%%%%%%%%%%%%%%%%%%%%%%%%%%%%%%%%%%%%%%%%%%%%%%%%%%%%%%%%%%%%%%%%%%%%%%%%%%%%%%%%%%%%%%%%%%%%%%%%%%%%
\begin{definition}[\cite{Baczynski_Jayaram_2008}, Definition 2.5.1]
\label{def_R_imp}
     A function $I : [0,1]^{2} \longrightarrow [0,1]$ is called an \emph{R- implication }if  there exists a t-norm $T$ such that 
		$$I(x,y)=\sup\{t \in [0,1]|T(x,t) \leq y \},  \qquad \ x,y \in [0,1].$$  
\end{definition} 

If $I$ is an $R$-implication generated from a t-norm $T$, then it is denoted by $I_{T}$ and $I_T\in\mathbb{I}$.
\begin{theorem}
\label{thm_FE_R_imp}
Let $I$ be an R-implication generated from a t-norm. Then $I$ satisfies the equation (\ref{eqn:IE}).
\end{theorem}

\begin{proof}
 Let $I$ be an R-implication generated from a t-norm. Then $I$ has both (IP) and (NP). Since $I$ has (IP), for  $x \leq y$, we have $I(x,y)=1$. So,  on the one hand  if $x \leq y$, we have $I(x,y)=1$ and on the other hand $I(I(y,x),I(x,y))=I(I(y,x),1)=1.$ When $x >y$, again  we have  $I(I(y,x),I(x,y))=I(1,I(x,y))=I(x,y)$ by using (NP) of $I$. Thus it is shown  that every R-implication satisfies the equation (\ref{eqn:IE}).
\end{proof}

%%%%%%%%%%%%%%%%%%%%%%%%%%%%%%%%%%%%%%%%%%%%%%%%%%%%%%%%%%%%%%%%%%%%%%%%%%%%%%%%%%%%%%%%%%%%%%%%%%%%%%%%%%%%%%%%
%%%%%%%%%%%%%%%%%%%%%%%%%%%%%%%%%%%%%%%%%%%%%%%%%%%%%%%%%%%%%%%%%%%%%%%%%%%%%%%%%%%%%%%%%%%%%%%%%%%%%%%%%%%%%%%%
\subsection{(S,N)-implications and the equation (\ref{eqn:IE}).}
%%%%%%%%%%%%%%%%%%%%%%%%%%%%%%%%%%%%%%%%%%%%%%%%%%%%%%%%%%%%%%%%%%%%%%%%%%%%%%%%%%%%%%%%%%%%%%%%%%%%%%%%%%%%%%%%
%%%%%%%%%%%%%%%%%%%%%%%%%%%%%%%%%%%%%%%%%%%%%%%%%%%%%%%%%%%%%%%%%%%%%%%%%%%%%%%%%%%%%%%%%%%%%%%%%%%%%%%%%%%%%%%%
\begin{definition}[\cite{Baczynski_Jayaram_2008}, Definition 2.4.1]
\label{def_SN_imp}
     A function $I: [0,1]^{2} \longrightarrow [0,1]$ is called an \emph{(S,N)-implication} if there exist a t-conorm $S$ and a fuzzy negation $N$ such that 
		\begin{equation}
		\label{eqn:S,N-implication}
		I(x,y)=S(N(x),y),  \qquad \ x,y \in [0,1].
		\end{equation}
\end{definition}
If $I$ is an $(S,N)$-implication then we will often denote it by $I_{S,N}$.  The family of  $(S,N)$-implications will be denoted by $\mathbb{I_{S,N}}.$

 \begin{theorem}
 \label{thm_S,N imp cases_FE}
 Let $I$ be an (S,N)-implication defined from a t-conorm $S$ and a fuzzy negation $N$. Then in the following cases $I$ satisfies the equation (\ref{eqn:IE}).
 \begin{enumerate}
 \renewcommand{\labelenumi}{(\roman{enumi})} 
 \item $S=\max$  
\item $S=S_{D}$ 
 \item The pair $(S,N)$ satisfies $S(N(x),x)=1$, for all $x \in [0,1]$.
 \end{enumerate}
 \end{theorem}
 
 \begin{proof}
 Let $I$ be an (S,N)-implication defined from a t-conorm $S$ and a fuzzy negation $N$.
 \begin{enumerate}
 \renewcommand{\labelenumi}{(\roman{enumi})}
 \item Let $S(x,y)= \max(x,y)$, i.e., $S$ is maximum t-conorm. Then $I$ is given by $I(x,y)= \max(N(x),y)$, for all $x,y \in [0,1].$
 Let $x,y \in [0,1]$ be fixed arbitrarily. Then we have two cases: $x \leq N(y), x \geq N(y). $
\begin{enumerate}
 \renewcommand{\labelenumi}{(\roman{enumi})}
 \item Let $x \leq N(y).$ Then $N(x) \geq N(N(y)).$ Now, 
 \begin{align*}
 I(I(y,x),I(x,y))&= I(\max(N(y),x), \max(N(x),y))\\
 &= I(N(y), \max(N(x),y))\\
 &= \max(N(N(y)), \max(N(x),y))\\
 &=\max(\max(N(x), N(N(y))),y)\\
 &=\max(N(x),y)=I(x,y)
 \end{align*}
 Thus in this case $I$ sayisfies the equation (\ref{eqn:IE}).
 \item Let $x \geq N(y).$ Now, 
  \begin{align*}
 I(I(y,x),I(x,y))&= I(\max(N(y),x), \max(N(x),y))\\
 &= I(x, \max(N(x),y))\\
 &= \max(N(x), \max(N(x),y))\\
 &=\max(N(x),y)=I(x,y)
 \end{align*}
 Thus in this case also $I$ satisfies the equation (\ref{eqn:IE}).
\end{enumerate}
\item Let $S=S_{D}$, the drastic t-conorm. Then $I$ can be given by
$$I(x,y)= \begin{cases}
y, & \text{if $x=1$}\\
N(x), & \text{if $y=0$}\\
1,& \text{otherwise}
\end{cases}$$
Now, it is very easy to check that this (S,N)-implication $I$ satisfies the equation (\ref{eqn:IE}).
\item Assume that the pair $(S,N)$ satisfies $S(N(x),x)=1$, for all $x \in [0,1]$. Then the $(S,N)$-implication satisfies \eqref{eqn:IP}. Since, every $(S,N)$-implication satisfies \eqref{eqn:NP}, it follows easily that $(S,N)$-implication satisfies \eqref{eqn:IE}.
 \end{enumerate} 
 \end{proof}

%%%%%%%%%%%%%%%%%%%%%%%%%%%%%%%%%%%%%%%%%%%%%%%%%%%%%%%%%%%%%%%%%%%%%%%%%%%%%%%%%%%%%%%%%%%%%%%%%%%%%%%%%%%%%%%%
%%%%%%%%%%%%%%%%%%%%%%%%%%%%%%%%%%%%%%%%%%%%%%%%%%%%%%%%%%%%%%%%%%%%%%%%%%%%%%%%%%%%%%%%%%%%%%%%%%%%%%%%%%%%%%%%
\subsection{QL-implications and the equation (\ref{eqn:IE}).}
%%%%%%%%%%%%%%%%%%%%%%%%%%%%%%%%%%%%%%%%%%%%%%%%%%%%%%%%%%%%%%%%%%%%%%%%%%%%%%%%%%%%%%%%%%%%%%%%%%%%%%%%%%%%%%%%
%%%%%%%%%%%%%%%%%%%%%%%%%%%%%%%%%%%%%%%%%%%%%%%%%%%%%%%%%%%%%%%%%%%%%%%%%%%%%%%%%%%%%%%%%%%%%%%%%%%%%%%%%%%%%%%%
\begin{definition}[\cite{Baczynski_Jayaram_2008}, Definition 2.5.1]
\label{def_SN_imp}
     A function $I: [0,1]^{2} \longrightarrow [0,1]$ is called an \emph{(QL)-operation} if there exist a t-norm $T$, t-conorm $S$ and a fuzzy negation $N$ such that 
		\begin{equation}
		\label{eqn:QL-implication}
		I(x,y)=S(N(x),T(x,y)),  \qquad \ x,y \in [0,1].
		\end{equation}
\end{definition}
If $I$ is a $QL$-implication then we will often denote it by $I_{T,S,N}$. 

\begin{lemma}
\label{lemma_QL_IP}
If $I$ is a QL-implication satisfying (IP), then it also satisfies the equation (\ref{eqn:IE}).
\end{lemma} 

\begin{proof}
 It is known that  if $I$ is a QL-implication, then it satisfies (NP). Corollary \ref{thm:OP_FE} will imply that a QL-implication with (IP) will satisfy the equation (\ref{eqn:IE}). 
\end{proof}

\begin{theorem}
\label{QL-imp_FE}
Let $I$ be a QL-implication defined from a t-norm $T$, t-conorm $S$ and a fuzzy negation $N$. Then in the following cases, $I$ satisfies the equation (\ref{eqn:IE}).
\begin{enumerate}
\renewcommand{\labelenumi}{(\roman{enumi})}
\item $S$ is a positive t-conorm(This implies that $N=N_{D_{2}}$)
\item $T=T_m$.
\item $S=S_{D}, N$ is any non vanishing negation and $T$ is a positive t-norm.
\end{enumerate}
\end{theorem}

\begin{proof} Let $I$ be a QL-implication defined from a t-norm $T$, t-conorm $S$ and a fuzzy negation $N$.  
\begin{enumerate}
\renewcommand{\labelenumi}{(\roman{enumi})}
\item  Let  $S$ is a positive t-conorm. Then from Proposition 2.6.7 in \cite{Baczynski_Jayaram_2008}, we see that $N=N_{D_{2}}$. Thus in this case, $I=I_{WB}$, which satisfies the equation (\ref{eqn:IE}).
\item Let $I$ be a QL-implication. Then we know that the pair (S,N) satisfies (LEM). If $T=T_{m}$ from Proposition 2.6.3 in \cite{Baczynski_Jayaram_2008}, the QL-implication will have (IP). Already we know that every QL-implication has (NP). Thus $I$ satisfies the equation (\ref{eqn:IE}).
\item From  Proposition 2.6.21 in \cite{Baczynski_Jayaram_2008} proof follows easily.
\end{enumerate}
\end{proof}
%%%%%%%%%%%%%%%%%%%%%%%%%%%%%%%%%%%%%%%%%%%%%%%%%%%%%%%%%%%%%%%%%%%%%%%%%%%%%%%%%%%%%%%%%%%%%%%%%%%%%%%%%%%%%%%%
%%%%%%%%%%%%%%%%%%%%%%%%%%%%%%%%%%%%%%%%%%%%%%%%%%%%%%%%%%%%%%%%%%%%%%%%%%%%%%%%%%%%%%%%%%%%%%%%%%%%%%%%%%%%%%%%
\subsection{f-generated implications and the equation (\ref{eqn:IE}).}
%%%%%%%%%%%%%%%%%%%%%%%%%%%%%%%%%%%%%%%%%%%%%%%%%%%%%%%%%%%%%%%%%%%%%%%%%%%%%%%%%%%%%%%%%%%%%%%%%%%%%%%%%%%%%%%%
%%%%%%%%%%%%%%%%%%%%%%%%%%%%%%%%%%%%%%%%%%%%%%%%%%%%%%%%%%%%%%%%%%%%%%%%%%%%%%%%%%%%%%%%%%%%%%%%%%%%%%%%%%%%%%%%
\begin{definition}[\cite{Baczynski_Jayaram_2008}, Definition 3.1.1]
  Let $f: [0,1] \longrightarrow [0,\infty ]$ be a strictly decreasing and continuous function with $f(1)=0$. The function $I: [0,1]^{2} \longrightarrow [0,1]$ defined by
	\begin{align}
 \label{eqn:F_Imp_Defn}
	I(x,y)=f^{-1}(x \cdot f(y)) \ ,  \qquad x,y \in [0,1] \ ,
\end{align}
  with the understanding $ 0 \cdot \infty =0 $, is called an \emph{f-implication}. If $I$ is an $f$-implication then it is denoted by $I_{f}$. The family of  $f$-impllications will be denoted by $ \mathbb{I_F}.$
\end{definition}

\begin{theorem}
No $f$ generated implication $I$ satisfies the equation (\ref{eqn:IE}).
\end{theorem}

\begin{proof}
 Let $I$ be an f-implication generated  from a strictly decreasing function  $f: [0,1] \rightarrow [0, \infty]$ s.t $f(1)=0$.  Now the equation $I(I(y,x),I(x,y))=I(x,y)$ has the following form
$$f^{-1} \left(f^{-1}(y \cdot f(x)) \cdot x \cdot f(y) \right )=f^{-1}(x \cdot f(y)), \forall x,y \in [0,1]. $$
Let $ x,y \notin \{0,1\}.$ Then we have 
$$f^{-1}(y \cdot f(x)) \cdot x \cdot f(y)=x \cdot f(y)$$ 
$$yf(x)=f(1)=0$$
$$i.e , either \ y=0 \ or \ f(x)=0.$$
$$i.e , either \ y=0 \ or \ x=1.$$
This is a contradiction. 
\end{proof}
%%%%%%%%%%%%%%%%%%%%%%%%%%%%%%%%%%%%%%%%%%%%%%%%%%%%%%%%%%%%%%%%%%%%%%%%%%%%%%%%%%%%%%%%%%%%%%%%%%%%%%%%%%%%%%%%
%%%%%%%%%%%%%%%%%%%%%%%%%%%%%%%%%%%%%%%%%%%%%%%%%%%%%%%%%%%%%%%%%%%%%%%%%%%%%%%%%%%%%%%%%%%%%%%%%%%%%%%%%%%%%%%%
\subsection{g-generated implications and the equation (\ref{eqn:IE}).}
%%%%%%%%%%%%%%%%%%%%%%%%%%%%%%%%%%%%%%%%%%%%%%%%%%%%%%%%%%%%%%%%%%%%%%%%%%%%%%%%%%%%%%%%%%%%%%%%%%%%%%%%%%%%%%%%
%%%%%%%%%%%%%%%%%%%%%%%%%%%%%%%%%%%%%%%%%%%%%%%%%%%%%%%%%%%%%%%%%%%%%%%%%%%%%%%%%%%%%%%%%%%%%%%%%%%%%%%%%%%%%%%%
\begin{definition}[\cite{Baczynski_Jayaram_2008}, Definition 3.2.1]
  Let $ g: [0,1] \longrightarrow [0, \infty ]$ be a strictly increasing and continuous function with $g(0)=0.$ The function $I: [0,1]^{2} \rightarrow [0,1]$ defined by $$I(x,y)=g^{(-1)}\left(\frac{1}{x} \cdot g(y)\right), \qquad x,y \in [0,1] \ ,$$ with the understanding $ \frac{1}{0}= \infty $ and $ \infty \cdot 0 = \infty $, is called a \emph{g-generated implication}, where the function $g^{(-1)}$ is the pseudo inverse of $g$ given by 
	
	$$g^{(-1)}(x)=\begin{cases}
g^{-1}(x), & \text{if $ x \in [0, \ g(1)]$} \ ,\\
1, & \text{if $ x \in [g(1),  \ \infty ]$} \ .
\end{cases}$$

\end{definition}

\begin{lemma}
\label{lemma_g_imp_FE_1}
Let $I$ be a g-generated implication. If $I$ satisfies the equation (\ref{eqn:IE}) then $g(1) < \infty$.
\end{lemma}

 \begin{proof}
 Let $I$ be a g-generated implication. Let us assume that $I$ satisfies the equation (\ref{eqn:IE}). Suppose that $g(1)=\infty$. i.e., $g^{(-1)}=g^{-1}$.
 
 Let $x,y \in (0,1)$. Now the equation $I(I(y,x),I(x,y))=I(x,y)$ will imply
 $$g^{-1}\left (\frac{1}{g^{-1}(\frac{1}{y}\cdot g(x))} \cdot \frac{1}{x} \cdot g(y)\right)=g^{-1}\left (\frac{1}{x} \cdot g(y)\right)$$
 $$i.e., \frac{1}{g^{-1} \left (\frac{1}{y}\cdot g(x) \right)} \cdot \frac{1}{x} \cdot g(y)=\frac{1}{x} \cdot g(y)$$
 $$i.e., g(1)=\infty =\frac{1}{y} \cdot g(x)$$
 $$i.e., either \ g(x) = \infty \ or \ y=0 $$
 $$i.e., either \ x=1 \ or \ y=0$$ which is a contradiction. Thus $g(1) < \infty$.
 \end{proof}
 
 \begin{lemma}
 \label{lemma_g_gen_imp_g_1_FE}
 If  a $g-$ generated fuzzy implication $I_{g}$ satisfies the equation (\ref{eqn:IE}), then the $g$-generated fuzzy implication $I_{g_{1}}$ also satisfies the equation (\ref{eqn:IE}) where $\displaystyle g_{1}(x)= \frac{g(x)}{g(1)}, \forall \ x \in [0,1].$  
 \end{lemma}

 \begin{proof} 
We know that if for some $c \in (0, \infty)$, $g_{1}(x)= c \cdot g(x), \forall \ x \in [0,1],$ then $I_{g}=I_{g_{1}}.$ Thus proof completes.
 \end{proof}

\begin{lemma}
If $I$ is a g-generated implication with (IP), then $I$ satisfies the equation (\ref{eqn:IE}).
\end{lemma}

 \begin{proof}
 Let $I$ be a g-generated implication with (IP). We know that every g-generated implication has (NP). Thus every g-generated implication with (IP) satisfies the equation (\ref{eqn:IE}).
\end{proof}

 \begin{example} 
 $I_{GG}$ is a g-generated implication with (IP). It satisfies the equation (\ref{eqn:IE}).
\end{example}

 \begin{theorem}
 \label{thm_FE_gen}
 Let $I_{g}$ be a $g$-generated fuzzy implication satisfying the equation (\ref{eqn:IE}). Then $g(x)= g(1) \cdot x, $ for all $x \in [0,1].$
 \end{theorem}

 \begin{proof}
 Let $I_{g}$ be a $g$-generated fuzzy implication satisfying the equation (\ref{eqn:IE}). From the lemma \ref{lemma_g_imp_FE_1}, it follows that $g(1) < \infty$.
 Let us define $\displaystyle g_{1}(x)= \frac{g(x)}{g(1)}.$ Clearly $g_{1}$ is also a $g$-generator and $I_{g}(x,y)=I_{g_{1}}(x,y)$, for all $x,y \in [0,1].$ Since $I_g$ satisfies the equation (\ref{eqn:IE}), $I_{g_{1}}$ also satisfies the same equation (\ref{eqn:IE}) from the lemma \ref{lemma_g_gen_imp_g_1_FE}. Moreover $I_{g_{1}}$ has the following form
 $$I_{g_{1}}(x,y)= g_{1}^{-1}\left (\frac{1}{x} \cdot g_{1}(y) \right ), \qquad x, y \in [0,1]$$
  Now, the equation  $$I_{g_{1}} (I_{g_{1}}(y,x), I_{g_{1}}(x,y)  )=I_{g_{1}}(x,y)$$ becomes
  \begin{equation}
  \label{eqn_g imp_FE}
  g_{1}^{-1} \left (\frac{1}{g_{1}^{-1}(\frac{1}{y} \cdot g_{1}(x))} \cdot \frac{1}{x} \cdot g_{1}(y)  \right )= g_{1}^{-1} \left(\frac{1}{x} \cdot g_{1}(y) \right ), \qquad x,y \in [0,1].
  \end{equation}
Let $x=y$ in the equation (\ref{eqn_g imp_FE}). Then the equation (\ref{eqn_g imp_FE}) becomes
$$g_{1}^{-1} \left(\frac{1}{g_{1}^{-1}(\frac{1}{x} \cdot g_{1}(x))} \cdot \frac{1}{x} \cdot g_{1}(x)  \right)= g_{1}^{-1} \left(\frac{1}{x} \cdot g_{1}(x) \right ), \qquad x \in [0,1]$$
$$i.e., \frac{1}{g_{1}^{-1} \left (\frac{1}{x} \cdot g_{1}(x) \right )} \cdot \frac{1}{x} \cdot g_{1}(x) = \frac{1}{x} \cdot g_{1}(x)$$
$$i.e., g_{1}^{-1} \left(\frac{1}{x} \cdot g_{1}(x) \right) =1$$
$$i.e., \frac{1}{x} \cdot g_{1}(x) =g_{1}(1)=1$$
$$i.e., g_{1}(x)=x \ or \ g(x)=g(1) \cdot x, \qquad x \in [0,1].$$
\end{proof}
 
 \begin{lemma}
 \label{lemma_g_imp_FE}
 Let $I_{g}$ be any $g$-generated fuzzy implication satisfying the equation (\ref{eqn:IE}). Then $g(1) < \infty$ and $g(x)=g(1) \cdot x, $ for all $x \in [0,1].$
 \end{lemma}

 \begin{proof}
 Proof follows from the lemma \ref{lemma_g_imp_FE_1} and the  theorem \ref{thm_FE_gen}.
\end{proof}
 
 Here we will make use of  a theorem from MB, BJ book on Fuzzy Implications (Theorem 3.2.9)
 
 \begin{theorem}
\label{thm:I_g:andOP}
If $g$ is a $g$-generator, then the following statements are equivalent:
\begin{enumerate}
\renewcommand{\labelenumi}{(\roman{enumi})}
\item
    $I_g$ satisfies (OP).
\item
    $g(1)<\infty$ and there exists a constant $c\in(0,\infty)$ such that $g(x)=c\cdot x$ for all $x\in[0,1]$.
\item
    $I_g$ is the Goguen implication $I_{\bf GG}$.
\end{enumerate}
\end{theorem}

\begin{theorem}
\label{thm:I_g:complete_sol}
If $g$ is a $g$-generator, then the following statements are equivalent:
\begin{enumerate}
\renewcommand{\labelenumi}{(\roman{enumi})}
\item
    $I_g$ satisfies the equation (\ref{eqn:IE}).
\item
    $I_g$ is the Goguen implication $I_{\bf GG}$.
\end{enumerate}
\end{theorem}

\begin{proof}
 $(i) \Rightarrow (ii):$ Proof of this part follows from Lemma \ref{lemma_g_imp_FE} and Theorem \ref{thm:I_g:andOP}.\\
$(ii) \Rightarrow (i): $ Since $I_{g}=I_{\bf{GG}}$ is an $R$-implication, $I_{g}$ satisfies the equation (\ref{eqn:IE}) from Theorem \ref{thm_FE_R_imp}.
\end{proof}

\section{Conclusion}
The functional equation (IE) has both theoretical and applicational significance in many areas. keeping the significance of this equation, here in this work, we have taken the task of finding the solutions of (IE). Interestingly, while every R-implication satisfies (IE), it is found out  that no f-implication satisfies the equation (IE). However, application of fuzzy implications satisfying (IE) are yet to be done.

%%%%%%%%%%%%%%%%%%%%%%%%%%%%%%%%%%%%%%%%%%%%%%%%%%%%%%%%%%%%%%%%%%%
%%  BIBLIOGRAPHY                                                 %%
%%%%%%%%%%%%%%%%%%%%%%%%%%%%%%%%%%%%%%%%%%%%%%%%%%%%%%%%%%%%%%%%%%%

%\begin{thebibliography}{99}
\bibliographystyle{elsarticle-harv} 
\bibliography{NRV_IE}

\end{document}